\documentclass{amsart}
\usepackage[utf8]{inputenc}
\usepackage{amsmath}
\usepackage{amssymb}
\usepackage{parskip}
\usepackage{mathtools}
\usepackage{enumitem,xcolor}
\usepackage{amsthm}
\usepackage[margin=2.5cm]{geometry}
\usepackage{hyperref}
\usepackage{mathrsfs}
\usepackage{graphicx}
\usepackage{multicol}
\usepackage{stmaryrd}
\usepackage{tikz-cd}
\usetikzlibrary{matrix,arrows,decorations.pathmorphing}
\usepackage{scalerel,stackengine}
\usepackage{xparse}
\usepackage[style=numeric,giveninits]{biblatex}

\makeatletter
\newenvironment{smallermatrix}[1][c]
{\null\,\vcenter\bgroup
  \Let@\restore@math@cr\default@tag
  \baselineskip0pt \lineskip0.4pt \lineskiplimit0pt
  \ialign\bgroup\if#1l\else\hfil\fi$\m@th\scriptstyle##$\if#1r\else\hfil\fi&&\thickspace\hfil
  $\m@th\scriptstyle##$\hfil\crcr
}{%
  \crcr\egroup\egroup\,%
}
\makeatother

\NewDocumentCommand{\ts}{O{c} e{^?_}}{
  \begin{smallermatrix}[#1]
  \mathstrut\IfValueT{#2}{#2} \\
  \mathstrut\IfValueT{#3}{#3} \\
  \mathstrut\IfValueT{#4}{#4}
  \end{smallermatrix}%
}

\newtheorem{thm}{Theorem}
\newtheorem{prop}[thm]{Proposition}
\newtheorem{lem}[thm]{Lemma}
\newtheorem{question}[thm]{Question}

\newtheorem{cor}[thm]{Corollary}

\newtheorem{conj}[thm]{Conjecture}

\stackMath
\newcommand\reallywidehat[1]{%
\savestack{\tmpbox}{\stretchto{%
  \scaleto{%
    \scalerel*[\widthof{\ensuremath{#1}}]{\kern-.6pt\bigwedge\kern-.6pt}%
    {\rule[-\textheight/2]{1ex}{\textheight}}
  }{\textheight}%
}{0.5ex}}%
\stackon[1pt]{#1}{\tmpbox}%
}

\DeclareFontFamily{U}{wncy}{}
    \DeclareFontShape{U}{wncy}{m}{n}{<->wncyr10}{}
    \DeclareSymbolFont{mcy}{U}{wncy}{m}{n}
    \DeclareMathSymbol{\Sh}{\mathord}{mcy}{"58}

\begin{document}
\title{Effective generation of Hecke algebras and explicit estimates of Sato--Tate type}
\author{Ben Moore}
\address{Mathematics Institute, University of Warwick, CV4 7AL, UK, and School of Mathematical and Computer Sciences,
University of Heriot--Watt, Edinburgh EH14 4AS, UK}
\date{December 2023}

\maketitle

\begin{abstract}
     Assuming the Riemann hypothesis for $L$-functions attached to primitive Dirichlet characters, modular cusp forms, and their tensor products and symmetric squares, we write down explicit finite sets of Hecke operators that span the Hecke algebras acting on the spaces of modular forms of weight greater than two with squarefree level.
\end{abstract}

\section{Introduction and overview}
Let $\mathscr{M}_{k}(N,\chi)$ denote the space of modular forms of weight $k$, level $N$ and nebentypus $\chi$. Then $\mathscr{M}_{k}(\Gamma_0(N))=\mathscr{M}_{k}(N,\mathbf{1}_{N})$, where $\mathbf{1}_{N}$ is the principal character of order $N$, and $\mathscr{M}_{k}(\Gamma_1(N))=\bigoplus_{\lvert\chi\rvert=N}\mathscr{M}_{k}(N,\chi)$. Let $\mathbb{T}_{k}(N,\chi)$ denote the $\mathbb{C}$-algebra of Hecke operators acting on the space $\mathscr{M}_{k}(N,\chi)$, let $\mathbb{T}_{k}(\Gamma_1(N))$ denote the algebra of Hecke operators on $\mathscr{M}_{k}(\Gamma_1(N))$, and write $\mathbb{T}_{k}(\Gamma_0(N))$ for $\mathbb{T}_{k}(N,\mathbf{1}_{N})$. For each of these Hecke algebras $\mathbb{T}$, we write $\mathbb{T}\vert_{\mathscr{S}}$ for the restriction of $\mathbb{T}$ to the subspace of forms vanishing at the cusp at infinity.

A famous theorem of Sturm \cite{sturm} asserts that if $\Gamma^{\prime}$ is a congruence subgroup of index $m^{\prime}$ in $SL(2,\mathbb{Z})$, then any $f \in \mathscr{M}_{k}(\Gamma^{\prime})$ with Fourier coefficients $a_j(f)=0$ for $j\leq km^{\prime}/12$ must be zero. This leads naturally to: 
\begin{question}\label{question}
    Suppose that for $f \in \mathscr{M}_{k}(\Gamma^{\prime})$, we know that $a_j(f)=0$ for $j=1,\dots ,n$. How big must $n$ be before we can conclude that $f=0$?
\end{question} 
This question admits a natural reinterpretation in terms of finding generators for Hecke algebras. Rewritten in terms of Hecke algebras and specialised to the case where $\Gamma^{\prime}$ is either $\Gamma_0(N)$ or $\Gamma_1(N)$, Sturm's bound implies that
\begin{equation*}
    \mathbb{T}_{k}(\Gamma^{\prime})\rvert_{\mathscr{S}}=\mathrm{Span}_{\mathbb{C}}\{T_1,\dots,T_{\lfloor km^{\prime}/12\rfloor}\}.
\end{equation*} Then Question \ref{question}, for these $\Gamma^{\prime}$, is equivalent to asking for an explicit presentation for the Hecke algebras $\mathbb{T}_k(N,\chi)$ and $\mathbb{T}_k(\Gamma_1(N))$, where the action is on the full space of modular forms of the indicated type. But in this case, far less is known. When $N=1$, we have \begin{equation*}
    \mathbb{T}_{k}(\Gamma_0(1)))=\mathrm{Span}_{\mathbb{C}}\left\{T_1,\dots, T_{\lfloor k/12 \rfloor}\right\}:
\end{equation*} this follows from the existence of the Miller basis. But neither this argument, nor an alternative due to Jenkins \cite{jenkins}, give any hint towards answering Question \ref{question} by finding an effective upper bound, the \emph{Hecke bound}, depending on $k$ and $N$, for $n$ such that \begin{equation}\label{eq:originalproblem}
    \mathbb{T}_{k}(N,\chi)=\mathrm{Span}_{\mathbb{C}}\{T_1,\dots, T_n\},
\end{equation} or towards proving the analogous statement for $\mathbb{T}_{k}(\Gamma_1(N))$.

The only publicly available speculation concerning the Hecke bound is buried in the source code of the function \texttt{hecke\_bound()} in Sage, which posits that we may take \begin{equation}\label{eq:sagebound}
    n=\lfloor km/12 \rfloor+2\mathrm{dim}\mathscr{E}_{k}(\Gamma_0(N))+5
\end{equation} for $\mathbb{T}_{k}(\Gamma_0(N))$, where $m=[SL(2,\mathbb{Z}):\Gamma_0(N)]$ and $\mathrm{dim}\mathscr{E}_{k}(\Gamma_0(N))$ is the dimension of the Eisenstein subspace. But we note that the documentation for the function points out that whilst its $n$ has never been observed to fail for $\mathbb{T}_{k}(\Gamma_0(N))$, it is not reliable for spaces of modular forms on $\Gamma_1(N)$ (if $m:=[SL(2,\mathbb{Z}): \Gamma_1(N)]$ and $\mathrm{dim}\mathscr{E}_{k}(\Gamma_0(N))$ is replaced by $\mathrm{dim}\mathscr{E}_{k}(\Gamma_1(N))$): for $\mathbb{T}_{4}(\Gamma_1(17))$ it returns $n=15$, which is incorrect. As far as the author is aware, no work has been done towards collating numerical evidence to determine of the true order of growth of the Hecke bound.

We provide a partial solution towards finding an upper bound for the Hecke bound by showing if the Riemann hypothesis is true for certain lists of $L$-functions, we can write down explicit sets of generators for $\mathbb{T}_{k}(\Gamma_0(N))$, at least when $N$ is squarefree and $k\geq 4$. We take this opportunity to define this list of $L$-functions: for each divisor $M$ of $N$, let $\mathscr{L}(M)$ denote the set consisting of \begin{enumerate}
    \item All the automorphic forms $\mathrm{Sym}^2(f)$ as $f$ ranges over Hecke eigenbases for $\mathscr{S}_{k}^{new}(\Gamma_0(d))$ and $d$ ranges over all divisors of $M$;
    \item All the automorphic forms $f\otimes g$ as $f$ and $g$ range over all distinct forms in Hecke eigenbases for $\mathscr{S}_{k}^{new}(\Gamma_0(d))$ and $d$ ranges over all divisors of $M$;
    \item All the primitive Dirichlet characters $\chi$ of conductor $d>1$ dividing $M$, and
    \item All the modular forms of the form $f\otimes \chi$, where $\chi$ runs over all primitive Dirichlet characters of conductor dividing $M$ and $f$ runs over Hecke eigenbases for $\mathscr{S}_{k}^{new}(\Gamma_0(d))$ as $d$ runs over all divisors of $M$. 
\end{enumerate} For each $\phi \in \mathscr{L}(M)$, we have an $L$-function, defined as follows (with the analytic normalisation): \begin{center}
\begin{tabular}{ c c } 
 
 $\phi$ & $L(\phi,s)$  \\ 
 \hline
 $\mathrm{Sym}^2 f$ & $\prod_{p}\left((1-\alpha_{f,1}(p)^2p^{-s})(1-\alpha_{f,1}\alpha_{f,2}(p)p^{-s})(1-\alpha_{f,2}(p)^2p^{-s})\right)^{-1}$  \\ 
 $f\otimes g$ & $\prod_{p}\prod_{1\leq i,j\leq 2}\left(1-\alpha_{f,i}(p)\alpha_{g,j}p^{-s}\right)^{-1}$  \\ 
 $\chi$ & $\prod_{p}\left(1-\chi(p)p^{-s}\right)^{-1}$ \\
 $f\otimes \chi$ & $\prod_{p}\left((1-\chi(p)\alpha_{f,1}(p)p^{-s})(1-\chi(p)\alpha_{f,2}(p)p^{-s})\right)^{-1}$ \\
 
\end{tabular}
\end{center} The Ramanujan conjecture guarantees that $|\alpha_{f,i}(p)|=1$ for $p\nmid N_f$.

\begin{conj}\label{conj:GRH}
    For every $\phi \in \mathscr{L}(M)$, the nontrivial zeros of $L(\phi,s)$ all have real part equal to $\frac{1}{2}$.
\end{conj}
Our main result is then as follows:
\begin{thm}\label{thm:main}
    Suppose that for every $M\mid N$, Conjecture \ref{conj:GRH} is true. Then \begin{equation*}
    \mathbb{T}_{k}(\Gamma_0(N))=\mathrm{Span}_{\mathbb{C}}\bigcup_{M\mid N}\left\{T_{pN/M} \text{ s.t. }p \nmid M, p \leq X_{M}\right\},
\end{equation*} whenever \begin{align*}
        X_{M}&\geq \max\{6(k+1)^2,M,17.33\},\\
    \frac{X_{M}}{\log^4(X_M)}&\geq 60 k (\alpha+1)^2 (\epsilon_{M}+1)(s_{M}+1)\log(\epsilon_{M}+3)\log(s_{M}+3),\\
    \frac{X_{M}^{1/2}}{\log^2(X)}&\geq \max\{12(\alpha+1)s_{M}\log(s_{M}+3),k\epsilon_{M} (\alpha+1)^2+1\},\\
    \frac{X_{M}^{k/2}}{\log^2(X_M)}&\geq 24k (\alpha+1)(\epsilon_{M}+1)\log(\epsilon_{M}+3).
\end{align*} where \begin{gather*}
        \alpha = 2605+87k+(248+6k)\log(M),\\
        s_M=\sum_{d\mid M}\mathrm{dim}\mathscr{S}_{k}^{new}(\Gamma_0(d)), \quad
        \epsilon_M=\sum_{d\mid M}\mathrm{dim}\mathscr{E}_{k}^{new}(\Gamma_0(d)).
    \end{gather*}
\end{thm} 
Noting that $X_M N/M\leq X_{N} N$, we immediately obtain a solution to the problem at (\ref{eq:originalproblem}), albeit with a weaker bound for $\Gamma_0(N)$ than the one hypothesised at (\ref{eq:sagebound}). 
\begin{cor}\label{cor:support}
    Suppose that for every $M \mid N$, Conjecture \ref{conj:GRH} is true. Then \begin{equation*}
        \mathbb{T}_{k}(\Gamma_0(N))=\mathrm{Span}_{\mathbb{C}}\left\{T_1,\dots,T_n\right\}
    \end{equation*} whenever $n \geq NX$ for any $X$ such that \begin{align*}
        X&\geq \max\{6(k+1)^2,N,17.33\},\\
    \frac{X}{\log^4(X)}&\geq 60 k (\alpha+1)^2 (\epsilon_{N}+1)(s_{N}+1)\log(\epsilon_{N}+3)\log(s_{N}+3),\\
    \frac{X^{1/2}}{\log^2(X)}&\geq \max\{12(\alpha+1)s_{N}\log(s_{N}+3),k\epsilon_{N} (\alpha+1)^2+1\},\\
    \frac{X^{k/2}}{\log^2(X)}&\geq 24k (\alpha+1)(\epsilon_{N}+1)\log(\epsilon_{N}+3).
\end{align*} where \begin{gather*}
        \alpha = 2605+87k+(248+6k)\log(N),\\
        s_{N}=\mathrm{dim}\mathscr{S}_{k}(\Gamma_0(N)), \quad
        \epsilon_{N}=\mathrm{dim}\mathscr{E}_{k}(\Gamma_0(N)).
    \end{gather*}
\end{cor}
The strategy of the proof is as follows. Section \ref{sec:reduction} is devoted to explaining how Theorem \ref{thm:main} ultimately follows from the effective upper bounds for certain partial sums of Fourier coefficients given in Theorem \ref{thm:effectiveST}). This is the heart of the proof. The link between the two theorems is that our problem is solved as soon as we find $X$ sufficiently large so that there exists a collection of integers $(n_i) \leq X$ with \begin{equation*}
    \mathrm{det}(a_{n_i}(f_j))\neq 0
\end{equation*}
for some basis $f_j$ of modular forms. We do this by showing that
\begin{equation}\label{eq:detsumfunction}
    \sum_{(n_i)\leq X}\left(\mathrm{det}(a_{n_i}(f_j))\right)^2
\end{equation}
is nonzero. But standard Tauberian techniques allow us to control such sums, as long as the $f_j$ are newforms and the $n_i$ are chosen so that only the (appropriately normalised) Fourier coefficients at primes $p_i$ contribute. This is possible to arrange by first reducing the problem to restrictions of Hecke algebras to certain subspaces of newforms. Then, upon expanding the determinant, we can identify the main term of (\ref{eq:detsumfunction}). One then need only verify that the error terms do in fact grow more slowly than the main term: this is where the Tauberian methods come in.

In subsequent sections, we generalise arguments due to Rouse--Thorner \cite{rousethorner} and prove that Conjecture \ref{conj:GRH} implies Theorem \ref{thm:effectiveST}. These arguments are quite standard Tauberian theorems, relating the asymptotic behaviour of certain Chebyshev theta functions to the distribution of the zeros of the associated $L$-functions. In fact, the core of the method goes back to Hadamard \cite{hadamard} and von Mangoldt \cite{Mangoldt}, who use it to prove an explicit version of the prime number theorem by producing an effective asymptotic for the Chebyshev theta function \begin{equation*}
    \psi(X)=\sum_{p^m \leq X}\log(p).
\end{equation*} For any $L$-function $L(s,\phi)=\sum_{n\geq 1}a_n n^{-s}$ with suitable properties (that is, possessing a functional equation and conjectured to satsify GRH) one can use these same methods to prove analogous asymptotics for the associated Chebyshev theta function. For example, Davenport's book \cite[Chapter 19]{davenport} provides an upper bound for \begin{equation*}
    \psi(\chi, X)=\sum_{p^m \leq X}\chi(p)\log(p),
\end{equation*} for $\chi$ a primitive Dirichlet character, and Rouse--Thorner \cite{rousethorner} treat the case \begin{equation*}
    \psi(\phi, X)=\sum_{p^m \leq X}a_p(\phi)\log(p),
\end{equation*} where $\phi=\mathrm{Sym}^k(f)$ for $f$ a holomorphic cusp form on $\Gamma_0(N)$ with $N$ squarefree. Now it is clear that the work of Rouse--Thorner could be adapted easily to prove similar results for other $\phi$, such as the forms \begin{equation*}
    \mathrm{Sym}^{k_1}(f_1)\otimes\dots\otimes \mathrm{Sym}^{k_n}(f_n),
\end{equation*} where the $f_1,\dots f_n$ are distinct cuspforms, but we only require $\mathrm{Sym}^2 f$ and $f_1\otimes f_2$. However, unlike in previous work, we also need to consider the shifted Chebyshev theta functions \begin{equation*}
     \psi(\phi, \ell, X)=\sum_{p^m \leq X}a_p(\phi)p^{\ell}\log(p)
\end{equation*} for certain choices of $\ell\geq 0$ with $\phi =\chi$ or a holomorphic cuspform $f$. Thus, we repeat the entire argument to deal with these cases which have not been previously covered in the literature.

At the cost of expanding the length of the argument, one could remove the assumptions that the level is squarefree and the weight is large and treat the Hecke algebras $\mathbb{T}_{k}(N,\chi)$ and $\mathbb{T}_{k}(\Gamma_1(N))$. To keep the paper from becoming too long we have not made any attempt to optimise the constants that appear in the main theorem.

Furthermore, it is clear that the results of this article generalise without too much trouble to provide effective sets of generators for Hecke algebras of modular forms over number fields, of non-integral weight, or indeed for Hecke algebras over any natural finite-dimensional space of automorphic forms. There should also be a mod $p$ analogue of these results, but we haven't investigated this.

It is possible to remove assumptions of GRH in Conjecture \ref{conj:GRH} by using unconditional results on zero-free regions, at the expense of worse lower bounds for the $X_M$. However, it should also be possible to avoid dealing with sums over Fourier coefficients of prime index at all: if one were to make fully explicit the constant and term magnitude of the error term appearing in Chandrasakharan--Narasimhan's result \cite{chandrasekharannarasimhan} \begin{equation*}
    \int_{0}^{X}\left\lvert\sum_{n\leq t}a_n(f)\right\rvert^2 dt=c_{f} X^{k+1/2}+O_{\epsilon}(X^{k+\epsilon}),
\end{equation*} (or more precisely, a suitable analogue for the sums appearing in the definitions of $F(X), G(X)$ and $H(X)$ appearing just before Proposition \ref{prop:determinant}), then one could presumably show (c.f. Proposition \ref{prop:determinant}, and using notation from Section \ref{sec:reduction}) that \begin{equation*}
\int_{[0,X]^{m}}\left\lvert\sum_{(n_{\delta})\leq (t_{\delta})}\mathrm{det}(E\mid S)(n_{\delta})\right\rvert^2 dt_1\dots dt_m>0
\end{equation*} for $X$ larger than some explicit function of all the relevant constants, without having to assume the Riemann hypothesis at all.

\section{Reduction to effective Sato--Tate theorems}\label{sec:reduction}
To begin with, we observe that the decomposition \begin{equation*}
    \mathscr{M}_{k}(\Gamma_0(N))=\bigoplus_{M\mid N}\alpha_{N/M}\bigoplus_{d\mid M}\mathscr{M}_{k}^{new}(\Gamma_0(d)),
\end{equation*} where $(\alpha_m f)(z)=f(mz)$, allows us to reduce the problem to the case of (not necessarily cuspidal) newforms. Indeed, suppose that for $X_M$ sufficiently large, we have 
\begin{equation}\label{eq:effectivenewforms}
\mathbb{T}\left(\bigoplus_{d\mid M}\mathscr{M}_{k}^{new}(\Gamma_0(d))\right)=\mathrm{Span}_{\mathbb{C}}\left\{T_p \text{ s.t. }p \nmid M, p\leq X_M\right\},
\end{equation} where the Hecke algebra on the left is the restriction of $\mathbb{T}_{k}(M)$ to the subspace indicated. Since \begin{equation*}
\mathrm{det}\begin{pmatrix}
    A & 0 \\
    C & D
\end{pmatrix}    =\mathrm{det}(A)\mathrm{det}(D)
\end{equation*} whenever $A$ and $D$ are square matrices, it follows from the duality with Fourier coefficients that \begin{equation}\label{eq:redtoXm}
    \mathbb{T}_{k}(\Gamma_0(N))=\mathrm{Span}_{\mathbb{C}}\bigcup_{M\mid N}\left\{T_{pN/M} \text{ s.t. }p \nmid M, p \leq X_{M}\right\},
\end{equation} whenever the $X_M$ are all large enough that (\ref{eq:effectivenewforms}) holds. 

Our strategy for finding an admissible lower bound for each $X_M$ is to show that when $X_M$ is sufficiently large, there exist primes $p_1,\dots,p_n \leq X_M$ not dividing $M$ such that \begin{equation}\label{eq:matrix}
\mathrm{det}(a_{p_i}(f_j))\neq 0, 
\end{equation} where the $f_j$ run over a basis of newforms for $\bigoplus_{d\mid M}\mathscr{M}_{k}^{new}(\Gamma_0(d))$. We actually show that for $X_M$ sufficiently large, the sum \begin{equation}\label{eq:determinant}
    \sum_{(p_i) \leq X_M} \mathrm{det}(a_{p_i}(f_j))^2
\end{equation}
is positive, because we can relate the asymptotic behaviour of this sum to more approachable conjectures on automorphic representations via Tauberian methods (as Serre did for the Sato--Tate conjecture). To understand the determininant at (\ref{eq:determinant}), we need to understand the Fourier coefficients of Eisenstein newforms. Thanks to our assumption concerning the weight, the Eisenstein newforms in $\mathscr{E}_{k}^{new}(\Gamma_0(d))$ are of the form \begin{equation*}
    E_{k}^{\psi,\phi}(z)=\delta_{\psi,\mathbf{1}_{1}}L(\phi,1-k)/2+\sum_{n=1}^{\infty}\sigma_{k-1}^{\psi,\phi}(n)q^n,\quad \sigma_{k-1}^{\psi,\phi}(n)=\sum_{ab=n}\psi(a)\phi(b)b^{k-1},
\end{equation*} where $\mathbf{1}_{m}$ is the principal Dirichlet character of order $m$, and $\psi$ and $\phi$ are primitive Dirichlet characters with $\psi\phi=\mathbf{1}_{d}$. We have \begin{equation}\label{eq:newformdecomp}
    \bigoplus_{d\mid M}\mathscr{M}_{k}^{new}(\Gamma_0(d))= \bigoplus_{(\psi,\phi)} \langle E_{k}^{\psi,\phi}  \rangle \oplus \bigoplus_{d\mid M} \mathscr{S}_{k}^{new}(\Gamma_{0}(d)),
\end{equation} where the first direct sum is over pairs $\psi$ and $\phi$ as above with $d$ ranging over all divisors of $M$.

From now on, let $\epsilon$ stand for the dimension of the subspace of Eisenstein newforms appearing on the right hand side of (\ref{eq:newformdecomp}), $s$ be the dimension of the space of cuspforms in (\ref{eq:newformdecomp}), and set $m=\epsilon+s$. Let $f_1,\dots f_s$ be a basis for $\bigoplus_{d\mid M} \mathscr{S}_{k}^{new}(\Gamma_{0}(d))$. For any multi-index $(p_j)$ of primes of length $m$, we define a matrix $S$, whose $(i,j)$th entry $S_{i,j}$ is $\hat{a}_{p_j}(f_i):=a_{p_j}(f_i)/p^{(k-1)/2}$. Similarly, we define a matrix $E$ whose $(i,j)$th entry is the normalised Fourier coefficient $\psi_{j}(p_i)p_i^{(1-k)/2}+\phi_{j}(p_i)p_{i}^{(k-1)/2}$ of the Eisenstein series attached to $(\psi_j,\phi_j)$. We can put $E$ and $S$ side by side to make a function that takes tuples $(p_j)$ of length $m$ to $m\times m$ block matrices $(E\mid S)(p_j)$. From now on a primed sum means that primes dividing $M$ are omitted: \begin{equation*}
    \sum_{p\leq X}^{\prime}F(p):=\sum_{\substack{p \leq X\\ p\nmid M}}F(p), 
\end{equation*}
and we define \begin{gather*}
    \pi_{M, \ell}(X)=\sum_{\substack{p\leq X}}^{\prime}p^{\ell},\quad \pi_{M}(X)=\pi_{M,0}(X),\quad
    E(X)=\pi_{M,1-k}(X)+2\pi_{M}(X).
\end{gather*} Then set \begin{gather*}
    F(X)=\max\left\{\left\lvert\sum_{\substack{p\leq X}}^{\prime}\chi(p)p^{\frac{k-1}{2}}\right\rvert \text{ s.t. }\chi \bmod M, 1<N_{\chi}\right\}\\
    G(X)=\max\left\{\left\lvert\sum_{\substack{p\leq X}}^{\prime}\hat{a}_{p}(f)\chi(p)p^{\frac{k-1}{2}}\right\rvert\text{ s.t. }\chi \bmod M, 1<N_{\chi}; f \in \bigcup_{d\mid M}\mathscr{S}_{k}^{new}(\Gamma_0(d))\right\}\\
    H(X)=\max\left\{ \left\lvert\sum_{p\leq X}^{\prime}U_2(\hat{a}_{p}(f)) \right\rvert \text{ s.t. }f \in \bigcup_{d\mid M}\mathscr{S}_{k}^{new}(\Gamma_0(d)), \qquad\qquad\qquad\qquad\qquad\right.\\
    \left.\qquad\qquad\qquad\qquad\qquad\left\lvert\sum_{p\leq X}^{\prime}\hat{a}_{p}(f)\hat{a}_{p}(g) \right\rvert \text{ s.t. }f\neq g \in \bigcup_{d\mid M}\mathscr{S}_{k}^{new}(\Gamma_0(d))\right\},
\end{gather*} where $U_2(T)=4T^2-1$ is the degree $2$ Chebyshev polynomial of the second kind. The heart of the argument is the proof of the following proposition, which consists of a lengthy computation and appears at the end of this section.
\begin{prop}\label{prop:determinant}
     Set \begin{equation*}
         Q(X)=(H(X)+\pi_{M}(X))^s-\pi_{M}(X)^s+s\log(s+3)\left(H(X)+\pi_{M}(X)\right)^{s-1}H(X).
     \end{equation*} Then \begin{align}
        &\left\lvert 4^s (m!)^{-1}\sum_{(p_{\delta})\leq X}^{\prime}\left\lvert\mathrm{det}(E\mid S)\right\rvert^2(p_{\delta})-\pi_{M}(X)^s\pi_{M,k-1}(X)^{\epsilon}\right\rvert \leq \pi_{M,k-1}(X)^{\epsilon}Q(X) \label{eq:propdetQ}\\
          &\qquad\qquad+{\epsilon}\log(\epsilon+3)(E(X)+\pi_{M,k-1}(X))^{\epsilon-1}(E(X)+F(X))Q(X)\label{eq:propdetEFQ}\\
         &\qquad\qquad+{\epsilon}\log(\epsilon+3)\pi_{M}(X)^s (E(X)+\pi_{M,k-1}(X))^{\epsilon-1}(E(X)+F(X))\label{eq:propdetEF}\\
         &\qquad\qquad+\epsilon\log(\epsilon+3)s\log(s+3)\notag \\
         &\qquad\qquad\quad\times(E(X)+\pi_{M,k-1}(X)^{\epsilon-1}(H(X)+\pi_{M}(X))^{s-1}(\pi_{M,(1-k)/2}(X)+G(X))^2\label{eq:propdetall},
     \end{align} as long as \begin{gather}
         E(X)+\pi_{M,k-1}(X)\geq 12 \epsilon (E(X)+F(X)),\label{eq:lowerboundEF}\\
         H(X)+\pi_{M}(X) \geq 12s H(X),\label{eq:lowerboundH}\\
                 (E(X)+\pi_{M,k-1}(X))(H(X)+\pi_{M}(X)) \geq 6 \epsilon s (\pi_{M,(1-k)/2}(X)+G(X))^2.\label{eq:lowerboundEHG}
     \end{gather}
\end{prop}

To show that the sum is nonzero for $X$ sufficiently large, we need to control the size of the error term. We already have \begin{equation*}
    E(X)\leq 0.18+2.52X/\log(X) \leq 3X/\log(X),
\end{equation*} using (\ref{eq:pi1kupper}), and it follows from Lemma \ref{lem:auxbounds} that \begin{equation*}
    \pi_{M,k-1}(X)\geq \frac{1}{4k}X^{k}/\log(X),
\end{equation*} whenever $X\geq \max\{17,1.67M\}$.
\begin{thm}\label{thm:effectiveST}
    Conjecture \ref{conj:GRH} implies that, for $X\geq 17.33$, \begin{gather*}
        F(X), G(X) \leq \alpha X^{k/2}\log(X),\qquad
        H(X) \leq \alpha X^{1/2}\log(X),
    \end{gather*} where \begin{equation*}
        \alpha = 2605+87k+(248+6k)\log(M).
    \end{equation*}
    
\end{thm} To prove this theorem, we follow the methods of Rouse--Thorner \cite{rousethorner}, who proved a conditional effective version of the Sato--Tate conjecture for cuspidal newforms of squarefree level. Subsequent sections are devoted to carrying out this process.

Under the assumption of Theorem \ref{thm:effectiveST}, we prove Theorem \ref{thm:main}.
\begin{proof}[Proof of Theorem \ref{thm:main}]
    Proposition \ref{prop:determinant} implies that there exists some tuple of primes $(p_{\delta})\leq X$ with \begin{equation*}
    \mathrm{det}(E\mid S)(p_{\delta})\neq 0
\end{equation*} whenever $X$ is sufficiently large that the following upper bounds hold, together with the lower bounds (\ref{eq:lowerboundEF}), (\ref{eq:lowerboundH}) and (\ref{eq:lowerboundEHG}): \begin{align}
    \epsilon\log(\epsilon+3)\left(E(X)+F(X)\right)\left(E(X)+\pi_{M,k-1}(X)\right)^{\epsilon-1}&\leq \pi_{M,k-1}(X)^{\epsilon}/4 \label{eq: upperboundF} \\ 
    s\log(s+3) H(X)\left(H(X)+\pi_{M}(X)\right)^{s-1}&\leq\pi_{M}(X)^s/4 \label{eq:upperboundH}\\
    \left(H(X)+\pi_{M}(X)\right)^{s}-\pi_{M}(X)^{s}&\leq\pi_{M}(X)^s/4\label{eq:upperboundotherH}\\
         \epsilon\log(\epsilon+3)s\log(s+3)(E(X)+\pi_{M,k-1}(X))^{\epsilon-1}\qquad&\notag \\
         \times(H(X)+\pi_{M}(X))^{s-1}(\pi_{M,(1-k)/2}(X)+G(X))^2&\leq \pi_{M,k-1}(X)^{\epsilon}\pi_{M}(X)^s/4\label{eq:upperboundSandG} 
\end{align} Throughout, we assume that $X\geq \max\{6(k+1)^2,M, 17.33\}$ so that we may apply Lemma \ref{lem:auxbounds} and Theorem \ref{thm:effectiveST}.

Let us begin with (\ref{eq: upperboundF}). We want to prove that \begin{equation*}
    \left(\frac{E(X)}{\pi_{M,k-1}(X)}+1\right)^{\epsilon-1}\leq \frac{1}{4\epsilon\log(\epsilon+3)}\frac{\pi_{M,k-1}(X)}{E(X)+F(X)},
\end{equation*} which follows from the inequalities \begin{equation*}
    \frac{E(X)}{\pi_{M,k-1}(X)}\leq \frac{1}{\epsilon},\quad \frac{\pi_{M,k-1}(X)}{E(X)+F(X)}\geq 4e \epsilon\log(\epsilon+3). 
\end{equation*} Using the bounds from Lemma \ref{lem:auxbounds} and Theorem \ref{thm:effectiveST}, it suffices to take \begin{gather*}
    X\geq (6k \epsilon)^{1/k},\quad X^{k/2}/\log^2(X) \geq 8k e (\alpha+1) \epsilon\log(\epsilon+3).
\end{gather*}

For (\ref{eq:upperboundH}) we need to show that \begin{equation*}
    \left(\frac{H(X)}{\pi_{M}(X)}+1\right)^{s-1}\leq \frac{1}{4s \log(s+3)}\frac{\pi_{M}(X)}{H(X)},
\end{equation*} for which it suffices to show that \begin{equation*}
    \frac{H(X)}{\pi_{M}(X)}\leq \frac{1}{s},\quad \frac{\pi_{M}(X)}{H(X)} \geq 4es \log(s+3).
\end{equation*} Assuming that $X\geq \max\{6(k+1)^2,M, 17.33\}$, we apply the bounds from Lemma \ref{lem:auxbounds} and Theorem \ref{thm:effectiveST} to show that these inequalities both follow from the condition \begin{equation*}
   X^{1/2}/\log^2(X)\geq 4e\alpha s\log(s+3).
\end{equation*}

Next, we attend to (\ref{eq:upperboundotherH}). It suffices to show that \begin{equation*}
    \left(\frac{H(X)}{\pi_{M}(X)}+1\right)^s\leq \frac{5}{4},
\end{equation*} but this follows from \begin{equation*}
    \frac{H(X)}{\pi_{M}(X)}\leq\frac{2}{9s}.
\end{equation*} Employing the bounds from Lemma \ref{lem:auxbounds} and Theorem \ref{thm:effectiveST}, we see that it suffices to demand \begin{equation*}
    X^{1/2}/\log^2(X)\geq 4.5\alpha s.
\end{equation*}

We turn to (\ref{eq:upperboundSandG}). We need to show that \begin{equation*}
    \left(\frac{E(X)}{\pi_{M,k-1}(X)}+1\right)^{\epsilon-1}\left(\frac{H(X)}{\pi_{M}(X)}+1\right)^{s-1}\frac{(0.85+G(X))^2}{\pi_{M,k-1}(X)}\leq \frac{\pi_{M}(X)}{4\epsilon s \log(\epsilon+3)\log(s+3)},
\end{equation*} for which it suffices to show that \begin{gather*}
    \frac{E(X)}{\pi_{M,k-1}(X)}\leq \frac{1}{\epsilon},\quad \frac{H(X)}{\pi_{M}(X)}\leq \frac{1}{s},\quad \pi_{M}(X)\geq 4e^2 \epsilon s\log(\epsilon+3)\log(s+3) \frac{(0.85+G(X))^2}{\pi_{M,k-1}(X)}.
\end{gather*} We apply the bounds from Lemma \ref{lem:auxbounds} and Theorem \ref{thm:effectiveST} to show that these inequalities follow from \begin{equation*}
    X^{k/2}/\log^2(X)\geq 6k \epsilon,\quad X^{1/2}/\log^2(X)\geq \alpha s,\quad X/\log^4(X)\geq 8ke^2 (\alpha+1)^2\epsilon s \log(\epsilon+3)\log(s+3).
\end{equation*}

Finally we work out conditions on $X$ such that the lower bounds (\ref{eq:lowerboundEF}), (\ref{eq:lowerboundH}) and (\ref{eq:lowerboundEHG}) hold. The bound (\ref{eq:lowerboundH}) is easy: it follows from \begin{equation}\label{eq:12sa}
    X^{1/2}/\log^2(X)\geq 12s+\alpha.
\end{equation} For (\ref{eq:lowerboundEF}), it suffices to show that \begin{equation*}
    X^{k/2}/\log^2(X)\geq 6kX^{1-k/2}/\log^2(X)+24k \epsilon\left(3X^{1-k/2}\log^2(X)+\alpha\right).
\end{equation*} But since $X\geq k^2$, one can check that this follows from \begin{equation*}
    X^{k/2}/\log^2(X)\geq 24k \epsilon (\alpha+1) +1.
\end{equation*} Lastly, the bound (\ref{eq:lowerboundEHG}) follows from \begin{equation*}
    \left(\frac{X^{1/2}}{2k\log^2(X)}-\frac{3X^{1/2-k}}{\log^2(X)}\right)\left(\frac{X^{1/2}}{\log^2(X)}-\alpha\right)\geq 6 \epsilon s (\alpha+0.85)^2.
\end{equation*} But using (\ref{eq:12sa}), it suffices to prove that \begin{equation*}
    \left(\frac{X^{1/2}}{2k\log(X)}-\frac{3X^{1/2-k}}{\log^2(X)}\right)\geq \frac{1}{2}\epsilon (\alpha+0.85)^2.
\end{equation*} This follows from the assumption \begin{equation*}
    X^{1/2}/\log^2(X)\geq \epsilon k (\alpha+1)^2+1.
\end{equation*}

So upon comparing all the lower bounds for $X$ that we accumulated along the way, we have proved that under Conjecture \ref{conj:GRH}, (\ref{eq:effectivenewforms}) holds whenever \begin{align*}
    X_{M}&\geq \max\{6(k+1)^2,M,17.33\},\\
    \frac{X_{M}}{\log^4(X_M)}&\geq 60 k (\alpha+1)^2 (\epsilon+1)(s+1)\log(\epsilon+3)\log(s+3),\\
    \frac{X_{M}^{1/2}}{\log^2(X)}&\geq \max\{12(\alpha+1)s\log(s+3),k\epsilon (\alpha+1)^2+1\},\\
    \frac{X_{M}^{k/2}}{\log^2(X_M)}&\geq 24k (\alpha+1)(\epsilon+1)\log(\epsilon+3).
\end{align*} Theorem \ref{thm:main} then follows from (\ref{eq:redtoXm}).
\end{proof}
\begin{proof}[Proof of Proposition \ref{prop:determinant}]
    Recall that we have defined functions $E$, $S$ and $(E\mid S)$ that take $\epsilon$-tuples of primes, $s$-tuples of primes and $m$-tuples of primes to $m\times \epsilon$, $m\times s$ and $m\times m$ matrices respectively. Let $\delta=\{1,\dots,m\}$, $u=\{1,\dots, \epsilon\}$ and $v=\{1,\dots,s\}$. For $\alpha \subseteq \delta$ of length $\epsilon$, we let $E_{\alpha}$ denote the $\epsilon\times \epsilon$ minor of $E$ formed by keeping only the $\alpha_1$st,..., $\alpha_{\epsilon}$st rows, and similarly for $S_{\beta}$.

    Now Laplace's generalised row expansion implies that \begin{equation*}
        \mathrm{det}(E\mid S)=\sum_{\substack{\alpha \sqcup \beta=\delta \\ |\alpha|=\epsilon}}*\mathrm{det}(E_\alpha)\mathrm{det}(S_{\beta}),
    \end{equation*} where $\sqcup$ reminds us that $\alpha$ and $\beta$ are disjoint, and $*$ is an unimportant sign. Suppose $\alpha_i\sqcup \beta_i=\delta$ for $i=1,2$, and let $\sigma_i: \alpha_i\rightarrow u$ and $\tau_i: \beta_i\rightarrow v$ denote bijections. Then \begin{align}
        &\mathrm{det}(E_{\alpha_1})\overline{\mathrm{det}(E_{\alpha_2})}\mathrm{det}(S_{\beta_1})\mathrm{det}(S_{\beta_2})\\
        &=\sum_{\substack{\sigma_1: \alpha_1\rightarrow u\\ \sigma_2: \alpha_2\rightarrow u}}\sum_{\substack{\tau_1: \beta_1\rightarrow u\\ \sigma_2: \beta_2\rightarrow u}} *\left(\prod_{\substack{a \in \alpha_1 \setminus \alpha_2}}E_{a,\sigma_1(a)}S_{a,\tau_2(a)}\prod_{a\in \alpha_2\setminus \alpha_1}\overline{E_{a,\sigma_2(a)}}S_{a,\tau_1(a)}\right.\notag\\
        &\left.\times\prod_{\substack{a \in \alpha_1\cap \alpha_2 \\ \sigma_1(a)=\sigma_2(a)}}\left\lvert E_{a,\sigma_1(a)}\right\rvert^2\prod_{\substack{a \in \alpha_1\cap \alpha_2 \\ \sigma_1(a)\neq \sigma_2(a)}}E_{a,\sigma_1(a)}\overline{E_{a,\sigma_2(a)}} \prod_{\substack{b\in \beta_1\cap \beta_2\\ \tau_1(b)=\tau_2(b)}}S_{b,\tau_1(b)}^2\prod_{\substack{b\in \beta_1\cap \beta_2\\ \tau_1(b)\neq \tau_2(b)}}S_{b,\tau_1(b)}S_{b,\tau_2(b)}\right),\label{eq:detexpansion}
    \end{align} where we have used the fact that \begin{gather*}
        \alpha_1\setminus \alpha_2= \beta_2\setminus \alpha_1,\quad \alpha_2\setminus \alpha_1= \beta_1\setminus \beta_2,\\
        \delta=(\alpha_1 \setminus \alpha_2) \sqcup (\alpha_1\cap \alpha_2) \sqcup (\beta_1 \cap \beta_2) \sqcup (\alpha_2\setminus \alpha_1).
    \end{gather*}
    Now, we have \begin{equation*}
        \left\lvert\mathrm{det}(E_\alpha)\right\rvert^2=\sum_{\sigma: \alpha\rightarrow u}\prod_{a \in \alpha}\left\lvert E_{a,\sigma(a)}\right\rvert^2+\sum_{\substack{\sigma_i: \alpha\rightarrow u\\ \sigma_1\neq \sigma_2}}\prod_{\substack{a \in \alpha \\ \sigma_1(a)=\sigma_2(a)}}\left\lvert E_{a,\sigma_1(a)}\right\rvert^2\prod_{\substack{a \in \alpha \\ \sigma_1(a)\neq \sigma_2(a)}}E_{a,\sigma_1(a)}\overline{E_{a,\sigma_2(a)}},
    \end{equation*} and we know that \begin{gather*}
        \sum_{p_a\leq X}^{\prime}\left\lvert E_{a,c}\right\rvert^2(p_a)\leq E(X)+\pi_{M,k-1}(X)\\
         \left\lvert\sum_{p_a\leq X}^{\prime} E_{a,c_1}(p_a)\overline{E_{a,c_2}(p_a)}\right\rvert\leq E(X)+F(X)
    \end{gather*} It follows that \begin{align}
        &\left\lvert\sum_{(p_{\alpha})\leq X}^{\prime}\left\lvert\mathrm{det}(E_\alpha)\right\rvert^2(p_{\alpha})-\epsilon!\pi_{M,k-1}(X)^{\epsilon}\right\rvert \notag\\
        &\leq \sum_{\substack{\sigma_i: \alpha \rightarrow u\\ \sigma_1 \neq \sigma_2}}\left(E(X)+\pi_{M,k-1}(X)\right)^{\mathrm{Eq}(\sigma_1,\sigma_2)}\left(E(X)+F(X)\right)^{\epsilon-\mathrm{Eq}(\sigma_1,\sigma_2)},\label{eq:EqsumE}
    \end{align} where, for two sets $D_1, D_2$ with $|D_i|=d$ and $|D_1\cap D_2| = \ell$, and bijections $\rho_i: D_i\rightarrow C$, we define \begin{equation*}
        \mathrm{Eq}(\rho_1,\rho_2)=\#\{x \in D_1\cap D_2 \text{ s.t. }\rho_1(x)=\rho_2(x)\}.
    \end{equation*} But some combinatorics reveals that, for $0\leq \gamma \leq \ell$, \begin{equation*}
        \#\{\rho_1,\rho_2 \text{ s.t. } \mathrm{Eq}(\rho_1,\rho_2)=\gamma\}=d!(d-\ell)!F(\ell,\gamma),
    \end{equation*} where the \begin{equation*}
        F(\ell,\gamma)=\binom{\ell}{\gamma}\left\lfloor\frac{(l-\gamma)!}{e}\right\rceil\leq \frac{\ell!}{\gamma!}
    \end{equation*} are the \emph{rencontres numbers}, counting the number of $\rho \in S_{\ell}$ with $\gamma$ fixed points. Thus, the right hand side of (\ref{eq:EqsumE}) is equal to \begin{equation}\label{eq:eisrencontres}
        {\epsilon}!\sum_{\gamma=0}^{\epsilon-1}F(\epsilon,\gamma)\left(E(X)+\pi_{M,k-1}(X)\right)^{\gamma}\left(E(X)+F(X)\right)^{\epsilon-\gamma}.
    \end{equation} For every $\ell\geq 1$ and $T\geq 12\ell$, we have the estimate \begin{equation}\label{eq:Recontressumestimate}
        \ell!\sum_{\gamma=0}^{\ell} T^{\ell}/\gamma! \leq \log(\ell+2)T^{\ell},
    \end{equation} which we apply to (\ref{eq:eisrencontres}) to produce, for $E(X)+\pi_{M,k-1}(X)\geq 12\epsilon(E(X)+F(X))$, \begin{align*}
        &\left\lvert(\epsilon! )^{-1}\sum_{(p_{\alpha})\leq X}^{\prime}\left\lvert\mathrm{det}(E_\alpha)\right\rvert^2(p_{\alpha})-\pi_{M,k-1}(X)^{\epsilon}\right\rvert \leq {\epsilon}\log(\epsilon+3)\left(E(X)+\pi_{M,k-1}(X)\right)^{\epsilon-1}(E(X)+F(X)).
    \end{align*} Similarly, for the determinants of the square minors of the matrix $S$ we have 
    \begin{equation*}
        \mathrm{det}(S_\beta)^2=\sum_{\tau:\beta\rightarrow v}\prod_{b\in \beta}S_{b,\tau(b)}^2+\sum_{\substack{\tau_i: \beta\rightarrow v\\ \tau_1\neq \tau_2}}\prod_{\substack{b\in \beta\\ \tau_1(b)=\tau_2(b)}}S_{b,\tau_1(b)}^2\prod_{\substack{b\in \beta\\ \tau_1(b)\neq \tau_2(b)}}S_{b,\tau_1(b)}S_{b,\tau_2(b)},
    \end{equation*} and we note that \begin{equation*}
        S_{b,\tau(b)}^2=\frac{1}{4}\left(U_2(S_{b,\tau(b)})+1\right),
    \end{equation*} so \begin{equation*}
        \sum_{\tau:\beta\rightarrow v}\prod_{b\in \beta}S_{b,\tau(b)}^2=4^{-s}s!+4^{-s}s!\sum_{\tau:\beta\rightarrow v}\sum_{\substack{T\subseteq \beta \\ |T|\geq 1}}U_2(S_{b_1,\tau(b_1)})\times \dots \times U_2(S_{b_{|T|},\tau(b_{|T|})}).
    \end{equation*} Then, since we know that \begin{equation*}
        \left\lvert\sum^{\prime}_{p_{b}\leq X} U_2(S_{b,c})(p_b)\right\rvert \leq H(X),
    \end{equation*} it follows upon using the facts about rencontres numbers again, together with (\ref{eq:Recontressumestimate}), that for $H(X)+\pi_{M}(X) \geq 12s H(X)$, we have \begin{align*}
        &\left\lvert\sum_{(p_{\beta})\leq X}^{\prime}\mathrm{det}(S_{\beta})^2(p_{\beta})-4^{-s}s!\pi_{M}(X)^s\right\rvert\\
        &\leq 4^{-s} \left(s!(H(X)+\pi_{M}(X))^s-s!\pi_{M}(X)^s+\sum_{\substack{\tau_i: \beta \rightarrow v\\ \tau_1 \neq \tau_2}}\left(H(X)+\pi_{M}(X)\right)^{\mathrm{Eq}(\tau_1,\tau_2)}H(X)^{s-\mathrm{Eq}(\tau_1,\tau_2)}\right)\\
        &\leq 4^{-s}s!\left((H(X)+\pi_{M}(X))^s-\pi_{M}(X)^s+s\log(s+3)\left(H(X)+\pi_{M}(X)\right)^{s-1}H(X)\right).
    \end{align*} The identity \begin{equation*}
        |xy-AB|\leq |x-A||y-B|+B|x-A|+A|y-B|,
    \end{equation*} along with the fact that the number of choices of $\alpha$ and $\beta$ such that $\alpha\sqcup\beta=\delta$ is $\binom{m}{s,\epsilon}$, then allows us to conclude that the quantity \begin{equation*}
        \left\lvert 4^{s}(m!)^{-1}\sum_{\alpha\sqcup\beta=\delta}\sum_{(p_{\delta}\leq X}^{\prime}\left\lvert\mathrm{det}(E_{\alpha})\right\rvert^2\mathrm{det}(S_{\beta})^2 (p_{\delta})-\pi_{M}(X)^s\pi_{M,k-1}(X)^{\epsilon}\right\rvert
    \end{equation*} is bounded by the expressions appearing at (\ref{eq:propdetEFQ}), (\ref{eq:propdetEF}), and the right hand side of (\ref{eq:propdetQ}).

    It remains to give an upper bound for \begin{equation*}
         \left\lvert\sum_{(p_{\delta})\leq X}^{\prime}\mathrm{det}(E_{\alpha_1})\overline{\mathrm{det}(E_{\alpha_2})}\mathrm{det}(S_{\beta_1})\mathrm{det}(S_{\beta_2})\right\rvert,
    \end{equation*} where we are now assuming that $\alpha_1\neq \alpha_2$ (so that $\beta_1\neq \beta_2$ too). We know that \begin{gather*}
        \left\lvert\sum_{p_a \leq X}^{\prime}E_{a,c}S_{a,c}\right\rvert \leq \pi_{M,(1-k)/2}(X)+G(X),\\
        \left\lvert\sum_{p_b\leq X}^{\prime} U_1(S_{b,c_1})U_1(S_{b,c_2})(p_b)\right\rvert \leq H(X),
    \end{gather*} where $U_1(T)=2T$ is the first degree Chebyshev polynomial of the second kind, so returning to (\ref{eq:detexpansion}), we find that \begin{align}
        &\left\lvert\sum_{(p_{\delta})\leq X}^{\prime}\mathrm{det}(E_{\alpha_1})\overline{\mathrm{det}(E_{\alpha_2})}\mathrm{det}(S_{\beta_1})\mathrm{det}(S_{\beta_2})\right\rvert \notag \\
        &\leq 4^{-s}\sum_{\sigma_i: \alpha_i\rightarrow u}\sum_{\tau_i: \beta_i\rightarrow v} \left(\pi_{M,(1-k)/2}+G(X)\right)^{|\alpha_1\setminus\alpha_2|+|\alpha_2\setminus \alpha_1|}\left(E(X)+\pi_{M,k-1}(X)\right)^{\mathrm{Eq}(\sigma_1,\sigma_2)}\notag \\
        &\qquad\qquad \times\left(E(X)+F(X)\right)^{\epsilon-\mathrm{Eq}(\sigma_1,\sigma_2)}\left(H(X)+\pi_M(X)\right)^{\mathrm{Eq}(\tau_1,\tau_2)}H(X)^{s-\mathrm{Eq}(\tau_1,\tau_2)}.\label{eq:sumdetES}
    \end{align} Let $\ell=|\alpha_1\cap \alpha_2|$. Then $|\alpha_1\setminus\alpha_2|+|\alpha_2\setminus \alpha_1|=2\epsilon-2\ell$. Using the facts about rencontres numbers again, we find that the right hand side at (\ref{eq:sumdetES}) is equal to \begin{align*}
        & 4^{-s}\epsilon! s! (\epsilon-\ell)!^2\left(\pi_{M,(1-k)/2}(X)+G(X)\right)^{2\epsilon-2\ell}\\
        &\qquad\times\sum_{\gamma=0}^{\ell}F(\ell,\gamma)\left(E(X)+\pi_{M,k-1}(X)\right)^{\gamma}\left(E(X)+F(X)\right)^{\ell-\gamma}\\
        &\qquad\times\sum_{k=0}^{s-\epsilon+\ell}F(s-\epsilon+\ell,k)\left(H(X)+\pi_{M}(X)\right)^k H(X)^{s-\epsilon+\ell-k}.
    \end{align*} By (\ref{eq:Recontressumestimate}), this expression is bounded above by \begin{align*}
        4^{-s}\epsilon!s!(\epsilon-\ell)!\log(\epsilon+3)\log(s+3)&\left(H(X)+\pi_{M}(X)\right)^s\\
        &\times (E(X)+\pi_{M,k-1})^{\ell}\left(\frac{(\pi_{M,(1-k)/2}(X)+G(X))^2}{H(X)+\pi_{M}(X)}\right)^{\epsilon-\ell}
    \end{align*} as long as \begin{gather*}
        E(X)+\pi_{M,k-1}(X)\geq 12\epsilon\left(E(X)+F(X)\right),\quad H(X)+\pi_{M}(X) \geq 12sH(X).
    \end{gather*}
    Now, \begin{equation*}
        \#\{\alpha_1,\alpha_2 \subseteq \delta \text{ s.t. } |\alpha_1|=|\alpha_2|=\epsilon, |\alpha_1\cap \alpha_2|=\ell\}=\binom{s+\epsilon}{\epsilon}\binom{\epsilon}{\ell}\binom{s}{\epsilon-\ell},
    \end{equation*} so we have \begin{align}
        &\left\lvert 4^{s}(m!)^{-1}\sum_{\substack{\alpha_i\sqcup \beta_i=\delta\\ \alpha_1\neq \alpha_2}}*\sum_{(p_{\delta})\leq X}^{\prime}\mathrm{det}(E_{\alpha_1})\overline{\mathrm{det}(E_{\alpha_2})}\mathrm{det}(S_{\beta_1})\mathrm{det}(S_{\beta_2})\right\rvert \leq\log(\epsilon+3)\log(s+3)(H(X)+\pi_{M}(X))^s\notag\\
        &\times\sum_{\ell=\max\{0,\epsilon-s\}}^{\epsilon-1}\frac{\epsilon!s!}{\ell! (s-\epsilon+\ell)!}(E(X)+\pi_{M,k-1})^{\ell}\left(\frac{(\pi_{M,(1-k)/2}(X)+G(X))^2}{H(X)+\pi_{M}(X)}\right)^{\epsilon-\ell}.\label{eq:ellsum}
    \end{align} But for $m\geq 1$, $k\geq 0$ and $T\geq 6m(m+k)$, we have the estimate \begin{equation*}
        m!(m+k)!\sum_{\ell=0}^{m}\frac{T^{\ell}}{\ell! (k+\ell)!}\leq \log(m+2)\log(m+k+2)T^m,
    \end{equation*} which tells us that the sum at (\ref{eq:ellsum}) is bounded by \begin{equation*}
        \log(\epsilon+3)\log(s+3)(E(X)+\pi_{M,k-1})^{\epsilon-1}\frac{(\pi_{M,(1-k)/2}(X)+G(X))^2}{H(X)+\pi_{M}(X)}
    \end{equation*} as long as \begin{equation*}
        (E(X)+\pi_{M,k-1}(X))(H(X)+\pi_{M}(X)) \geq 6 \epsilon s (\pi_{M,(1-k)/2}(X)+G(X))^2.
    \end{equation*} This accounts for the error term at (\ref{eq:propdetall}).
\end{proof}

\section{The density of nontrivial zeros of $L$-functions}
Recall the set $\mathscr{L}(M)$ from the introduction. For each $\phi \in \mathscr{L}(M)$, we define the following quantities \begin{center}
\begin{tabular}{ c c c c c}
 $\phi$ & $Q(\phi,s)$ & $\Gamma(\phi,s)$ & $\mu_{\phi}$ & $q_{\phi}$\\
 \hline
 $\mathrm{Sym}^2f$ & $N_{f}^s$ & $\Gamma_{\mathbb{R}}(s+1)\Gamma_{\mathbb{C}}(s+k-1)$ & $3$ & $\log (N_f)$\\  
 $f\otimes g$ & $N_{f}^{s}N_{g}^{s}$ & $\Gamma_{\mathbb{C}(s)}\Gamma_{\mathbb{C}}(s+k-1)$  & $4$ & $\log(N_f N_g)$ \\
 $\chi$ & $N_{\chi}^{(s+\delta_{\chi})/2}$ & $\Gamma_{\mathbb{R}}(s+\delta_{\chi})$ & $1$ & $
 \frac{1}{2}\log(N_{\chi})$ \\
 $f\otimes \chi$ & $N_{f\otimes \chi}^{s/2}$ & $\Gamma_{\mathbb{C}}(s+(k-1)/2)$ & $2$ & $\frac{1}{2}\log(N_{f\otimes \chi})$
\end{tabular}
\end{center} where \begin{equation*}
    \Gamma_{\mathbb{R}}(s)=\pi^{-s/2}\Gamma(s/2),\quad\Gamma_{\mathbb{C}}(s)=2(2\pi)^{-s}\Gamma(s),
\end{equation*} so that we may define the completed $L$-function \begin{equation*}
    \Lambda(\phi,s):=Q(\phi,s)\Gamma(\phi,s)L(\phi,s).
\end{equation*}
It satisfies the functional equation
\begin{equation*}
    \Lambda(\phi,s)=\epsilon(\phi)\Lambda(\phi,1-s),
\end{equation*} for some complex number $\epsilon(\phi)$ of absolute value $1$. Note that $q_{\phi}$ is simply the constant value of $d\log Q(\phi,s)$.
\begin{lem}\label{lem:vonMangoldt}
    For $s=\sigma+it$ with $\sigma>1$, \begin{equation}
    \lvert d\log L(\phi,s)\rvert \leq 4\lvert \zeta^{\prime}(\sigma)\rvert/\zeta(\sigma).
\end{equation}
\end{lem}
\begin{proof}
    It suffices to observe that for $\lvert A_{\phi}(j) \rvert \leq \mu_{\phi} \Lambda(j)$, where $\Lambda(j)$ is the von Mangoldt function and \begin{equation*}
    d\log L(\phi,s)=\sum_{j=1}^{\infty}A_{\phi}(j)j^{-s}.\qedhere
    \end{equation*}
\end{proof}
\begin{lem}\label{lem:gammabound}
    For all $s=\sigma+iT\in \mathbb{C}$ with $T\geq 1$, and all $\phi \in \mathscr{L}(M)$, \begin{equation}\label{eq:gammabound}
        \lvert d\log \Gamma(\phi,s) \rvert\leq 13.18+4\log(|s|+k+2).
    \end{equation}
\end{lem}
\begin{proof}
    Let $\langle s \rangle = \mathrm{min}\{|s+j| \text{ s.t. }j \in \mathbb{Z}\cap [0,\infty)\}$. Use the formulae in the table at the start of the section for the gamma factors $\Gamma(\phi,s)$ together with the bound \begin{equation*}
        \lvert d\log \Gamma(s) \rvert \leq 4+2/\langle s\rangle+2\log(|s|+3)
    \end{equation*} from \cite[Lemma 5.2]{rousethorner} to work out individual bounds for each choice of $\phi$, and observe that the expression at (\ref{eq:gammabound}) is larger than each of these bounds. 
\end{proof}
Define, for $T>0$ and $j$ a nonnegative integer, \begin{gather*}
    N(\phi,T)=\#\{\rho=\frac{1}{2}+i\gamma \text{ s.t. }L(\phi,\rho)=0, \lvert \gamma-T\rvert<1\},\\
    N^{*}(\phi,j)=\#\{\rho=\frac{1}{2}+i\gamma \text{ s.t. }L(\phi,\rho)=0, j<\lvert \gamma\rvert\leq j+1\}.
\end{gather*} 
 Beginning with the Hadamard product \begin{equation}
    \Lambda(\phi,s)=\exp\left(a_{\phi}+sb_{\phi}\right)\prod_{\rho\neq 0,1}\left(1-\frac{s}{\rho}\right)\exp\left(s/\rho\right)
\end{equation}
and taking the logarithmic derivative of both sides gives the identity \begin{equation}\label{eq:logderiv}
    -d\log L(\phi,s)=q_{\phi}+d\log \Gamma(\phi,s)-b_{\phi}-\sum_{\rho\neq 0,1}\left(\frac{1}{s-\rho}-\frac{1}{\rho}\right).
\end{equation} It's known \cite{murty} that \begin{equation*}
    \mathrm{Re}(b_{\phi})=-\sum_{\rho}\mathrm{Re}(\rho^{-1}),
\end{equation*} so that \begin{equation}
    \sum_{\substack{\rho \text{ nontrivial}}} \mathrm{Re}\left(\frac{1}{s-\rho}\right)=q_{\phi}+\mathrm{Re}\left(d\log\Gamma(\phi,s)\right)+\mathrm{Re}\left(d\log L(\phi,s)\right).
\end{equation} 
\begin{lem}\label{lem:Nbounds}
    For any $|T|\geq 1$ and $j$ a nonnegative integer, \begin{gather*}
        N(\phi,T)\leq 33.5+4.34\log(M)+8.67\log(|T|+k+4),\\
        N^{*}(\phi,j)\leq 25.77+3.34\log(M)+6.67\log(j+k+4.5).
    \end{gather*}
\end{lem}\begin{proof}
    Set $s_0=2+iT$ and $\rho=\frac{1}{2}+i\gamma$. If $\lvert \gamma-T\rvert \leq 1$, then \begin{equation*}
        \mathrm{Re}\left(\frac{1}{s_0-\rho}\right)\geq \frac{6}{13},
    \end{equation*} so \begin{equation*}
        N(\phi,T)\leq \frac{13}{6}\sum_{\substack{\rho \\ |\gamma-T|\leq 1}}\mathrm{Re}\left(\frac{1}{s_0-\rho}\right)\leq \frac{13}{6}\sum_{\substack{\rho}}\mathrm{Re}\left(\frac{1}{s_0-\rho}\right).
    \end{equation*} The estimate follows upon noting that $\langle s\rangle \geq 1$ by the assumption on $T$, and applying Lemma \ref{lem:gammabound} to bound $\lvert d\log \Gamma(\phi,s_0)\rvert$ and $\lvert d\log \Gamma(\phi,1-s_0)\rvert$. By the table, $q_{\phi}\leq 2\log(M)$ for all $\phi$.

    To estimate $N^{*}(\phi,j)$, suppose that $T-\frac{1}{2}\leq \gamma \leq T+\frac{1}{2}$, so that \begin{equation*}
        \mathrm{Re}\left(\frac{1}{s_0-\rho}\right)\geq \frac{3}{5}.
    \end{equation*} Then set $T=j+\frac{1}{2}$ so that \begin{equation*}
        N^{*}(\phi,j)\leq \frac{5}{3}\sum_{\substack{\rho \\ j< |\gamma|\leq j+1}}\mathrm{Re}\left(\frac{1}{s_0-\rho}\right)\leq \frac{5}{3}\sum_{\substack{\rho}}\mathrm{Re}\left(\frac{1}{s_0-\rho}\right).
    \end{equation*} The estimate follows from using Lemma \ref{lem:gammabound} again in the same manner.
\end{proof}
\begin{lem}\label{lem:verticalstrip}
    If $s=\sigma+iT$, with $-\frac{1}{4}\leq \sigma \leq 3$ and $T\geq 1$,
    \begin{equation*}
        \left\lvert d\log L(\phi,s)-\sum_{|\gamma-T|\leq 1}\frac{1}{s-\rho} \right\rvert \leq 103.48+10.87\log(M)+29.74\log(|T|+k+5).
    \end{equation*}
\end{lem} 
\begin{proof}
    First, observe that \begin{equation}\label{eq:sum1}
        \sum_{|\gamma-T|\leq 1}\left\lvert\frac{1}{3+iT-\rho}\right\rvert\leq \frac{1}{2}N(\phi,T)\leq 16.75+2.17\log(M)+4.34\log(|T|+k+5),
    \end{equation} where we have used the upper bound for $N(\phi,T)$ in Lemma \ref{lem:Nbounds}. Also, \begin{equation*}
        \sum_{|\gamma-T|\geq 1}\left\lvert\frac{1}{s-\rho}-\frac{1}{3+iT-\rho}\right\rvert \leq 3\sum_{|\gamma-T|\geq 1}\frac{1}{1+|\gamma-T|^2}\leq \frac{87}{20}\sum_{|\gamma-T|\geq 1}\frac{5/2}{25/4+|\gamma-T|^2}.
    \end{equation*} By the GRH, we may assume that \begin{equation*}
        \frac{5/2}{25/4+|\gamma-T|^2}=\mathrm{Re}\left(\frac{1}{3+iT-\rho}\right),
    \end{equation*} and apply (\ref{eq:logderiv}) with $s=3+iT$ to produce the upper bound \begin{equation}\label{eq:sum2}
        \sum_{|\gamma-T|\geq 1}\left\lvert\frac{1}{s-\rho}-\frac{1}{3+iT-\rho}\right\rvert\leq 60.2+8.7\log(M)+17.4\log(|T|+k+5).
    \end{equation} Now, combining (\ref{eq:logderiv}) with the logarithmic derivative of the functional equation for $L(\phi,s)$, \begin{equation}\label{eq:logderivL}
        -d \log L(\phi,s)=2q_{\phi}+d\log L(\phi,1-s)+d\log \Gamma(\phi,s) +d \log \Gamma(\phi,1-s),
    \end{equation} we obtain \begin{equation*}
        d\log L(\phi,s)=b_{\phi}+\sum_{\rho}\left(\frac{1}{s-\rho}+\frac{1}{\rho}\right)-q_{\phi}-d\log \Gamma(\phi,s).
    \end{equation*} We can evaluate this expression at $s=\sigma+iT$ and $3+iT$, then subtract the results from each other to obtain \begin{equation*}
        d\log L(\phi,s)-d \log L(\phi,3+iT)=-d\log \Gamma(\phi,s)+d\log \Gamma(\phi,3+iT)+\sum_{\rho}\left(\frac{1}{s-\rho}-\frac{1}{3+iT-\rho}\right).
    \end{equation*} Using Lemma \ref{lem:gammabound} and Lemma \ref{lem:vonMangoldt}, together with (\ref{eq:sum1}) and (\ref{eq:sum2}), we obtain the desired upper bound.
\end{proof}

\section{Estimates for various contour integrals}
\begin{lem}\label{lem:leftwardLbound}
    If $s=\sigma+iT$ with $\sigma\leq -1/4$ and $T\geq 1$ then \begin{equation*}
        \lvert d\log L(\phi,s) \rvert\leq 40.23+4\log(M)+8\log(|\sigma|+|T|+k+3).
    \end{equation*}
\end{lem}
\begin{proof}
    Since $\mathrm{Re}(1-s)\geq 5/4$, we have \begin{equation*}
        |d\log L(\phi,1-s)|\leq 4|\zeta'(5/4)|/\zeta(5/4)
    \end{equation*} by Lemma \ref{lem:vonMangoldt} and the values for $\mu_{\phi}$ in the table. On the other hand, $\langle s\rangle \geq 1$ by the assumption on $T$, we can apply Lemma \ref{lem:gammabound} to bound $\lvert d\log \Gamma(\phi,s)\rvert$ and $\lvert d\log \Gamma(\phi,1-s)\rvert$. The claim follows upon using (\ref{eq:logderivL}).
\end{proof}
Now let $U+\frac{1}{4}>0$ be a large integer. Set $\sigma_0=1+\log(X)^{-1}$. Let $\gamma(T,U)=S_1\cup S_2 \cup S_3$, where \begin{gather*}
    S_1=\{-U+it \text{ s.t. } |t|\leq T\},\\
    S_2=\{\sigma \pm iT \text{ s.t. }-U\leq \sigma \leq -1/4\},\\
    S_3=\{\sigma \pm iT \text{ s.t. }-1/4\leq \sigma\leq \sigma_0\}.
\end{gather*} 
\begin{lem}
    If $2\leq T\leq X$, \begin{equation*}
        \lim_{U\rightarrow \infty}\int_{S_1}d\log L(\phi, s)\frac{X^s}{s+\ell}ds=0.
    \end{equation*}
\end{lem}
\begin{proof}
    This follows from arguments similar to those of \cite{murty}.
\end{proof}
The rest of this section is devoted to finding an upper bound for \begin{equation*}
    \lim_{U\rightarrow \infty}\int_{S_2\cup S_3}d\log L(\phi, s)\frac{X^s}{s+\ell}ds.
\end{equation*} set $s=\sigma+iT$ and define \begin{gather*}
    I_1(\phi,X,T,U)=\frac{1}{2\pi i}\int_{-U}^{-1/4}\left(\frac{X^{\bar{s}}}{\bar{s}+\ell}d\log L(\phi,\bar{s})-\frac{X^{s}}{s+\ell}d\log L(\phi,s)\right)d\sigma,\\
    I_2(\phi,X,T)=\frac{1}{2\pi i}\int_{-1/4}^{\sigma_0}\left(\frac{X^{\bar{s}}}{\bar{s}+\ell}d\log L(\phi,\bar{s})-\frac{X^{s}}{s+\ell}d\log L(\phi,s)\right)d\sigma.
\end{gather*} \begin{lem}
    For $T\geq 1$, $\lvert\lim_{U\rightarrow \infty}I_1(\phi,X,T,U)\rvert$ is bounded above by \begin{equation*}
        \frac{15.36+1.28\log(M)+2.55\log(|T|+k+3.25)}{T X^{1/4}\log(X)}.
    \end{equation*}
\end{lem} \begin{proof}
    This is straightforward, using the fact that $\lvert X^s/(s+\ell)\rvert \leq X^\sigma/T$ and bounding the $d\log L(\phi,s)$ terms by Lemma \ref{lem:leftwardLbound}. It is necessary to use the estimate (for $\alpha>3/4$) \begin{equation*}
        \int_{\frac{1}{4}}^{\infty}X^{-s}\log(s+\alpha)ds\leq\frac{\log(\alpha+1/4)+1}{X^{1/4}\log(X)}.\qedhere
    \end{equation*}
\end{proof}
\begin{lem}
    For $T\geq 1$, $\lvert I_2(\phi,X,T)\rvert$ is bounded above by 
    \begin{align*}
        \left((1.74\ell+7.01)\frac{X}{T}+21.75X^{1/2}\right)&(33.5+4.34\log(M)+8.67\log(|T|+k+4))\\
        +\frac{X}{T\log(X)}&\left(89.54+9.4\log(M)+51.47\log(|T|+k+5+\sigma_0)\right).
    \end{align*}
\end{lem} \begin{proof}
    Let $s=\sigma+iT$ as usual. The estimate in Lemma \ref{lem:verticalstrip} implies that \begin{align}
        &\left\lvert I_2(\phi,X,T)-\frac{1}{2\pi i}\int_{-1/4}^{\sigma_0}\left(\frac{X^{\bar{s}}}{\bar{s}+\ell}\sum_{|\gamma+T|\leq 1}\frac{1}{\bar{s}-\rho}-\frac{X^s}{s+\ell}\sum_{|\gamma-T|\leq 1}\frac{1}{s-\rho}\right)d\sigma \right\rvert \notag \\
        &\qquad\leq \frac{X}{T\log(X)}\left(89.54+9.4\log(M)+51.47\log(|T|+k+5+\sigma_0)\right),\label{eq:I2-integralbound}
    \end{align} 
    where we make use of the estimate, for $\alpha\geq -1/4$, \begin{equation*}
        \int_{-\frac{1}{4}}^{\sigma_{0}}X^{\sigma}\log(\sigma+\alpha)d\sigma \leq \frac{2X^{\sigma_0}\log(\alpha+\sigma_0)}{\log(X)}
    \end{equation*} too. But \begin{align*}
    \int_{-1/4}^{\sigma_0}\frac{X^{\sigma+iT}}{(\sigma+iT)(\sigma+iT-\rho)}d\sigma-&\int_{-1/4}^{\sigma_{0}}\frac{X^{\sigma+iT}}{(\sigma+iT+\ell)(\sigma+iT-\rho)}d\sigma\\
    &= \int_{-1/4}^{\sigma_0}\frac{-\ell X^{\sigma+iT}}{(\sigma+iT)(\sigma+iT+\ell)(\sigma+iT-\rho)}d\sigma,
    \end{align*} and the right hand side is bounded above by \begin{equation}\label{eq:residuethmbound}
        \ell \left\lvert X^{\sigma_0}\int_{T}^{\infty}\frac{1}{(\sigma_0+it+\ell)(\sigma_0+it)(\sigma_0+it-\rho)}dt\right\rvert+8\pi \ell X^{1/2}
    \end{equation} by the residue theorem. Combining the upper bound $2e\ell(X/T+4\pi X^{1/2})$ for (\ref{eq:residuethmbound}) with the bound \begin{equation*}
        \left\lvert\int_{-1/4}^{\sigma_0}\frac{X^{\sigma+iT}}{(\sigma+iT)(\sigma+iT-\rho)}\right\rvert \leq 22X/T
    \end{equation*} from the proof of Lemma 6.4 in \cite{rousethorner}, we obtain \begin{equation}\label{eq:intbound}
        \left\lvert\int_{-1/4}^{\sigma_0}\frac{X^{\sigma+iT}}{(\sigma+iT+\ell)(\sigma+iT-\rho)}\right\rvert \leq (2e\ell+22)X/T+8\pi e\ell X^{1/2}.
    \end{equation} Then (\ref{eq:I2-integralbound}) and (\ref{eq:intbound}) together imply \begin{equation*}
        |I_2|\leq 2N(\phi,T)\left(\frac{e\ell+11}{\pi}\frac{X}{T}+4e\ell X^{1/2}\right)
    \end{equation*} and the proof is concluded by an appeal to Lemma \ref{lem:Nbounds}.
\end{proof}
\section{Estimates for truncated Perron integrals}
In this section we provide upper bounds for the integrals \begin{equation*}
    \frac{1}{2\pi i}\int_{\sigma_0-iT}^{\sigma_0+iT}d\log (\phi,s)\frac{X^s}{X+\ell}ds.
\end{equation*} By the residue theorem, this is equal to \begin{equation*}
    \frac{1}{2\pi i}\int_{\gamma(T,U)}d\log (\phi,s)\frac{X^s}{X+\ell}ds+\sum_{\substack{\rho=\frac{1}{2}+i\gamma\\ |\gamma|\leq T}}\frac{{X}^{\rho}}{\rho+\ell}+\sum_{\rho \text{ trivial}}\frac{{X}^{\rho}}{\rho+\ell}+\underset{s=-\ell}{\mathrm{Res}}\left(d\log L(\phi,s)\frac{X^{s}}{s+\ell}\right),
\end{equation*} where both sums over $\rho$ are over zeros of $L(\phi,s)$, and we already an upper bound for the integral over $\gamma(T,U)$ . From \cite[Section 7.2]{rousethorner} we have \begin{align*}
    &\sum_{\substack{\rho=\frac{1}{2}+i\gamma\\ |\gamma|\leq T}}\left\lvert\frac{{X}^{\rho}}{\rho+\ell}\right\rvert\leq 2\sqrt{X}\sum_{j=0}^{3}N^{*}(\phi,j)+\sqrt{X}\sum_{4\leq j \leq T}N^{*}(\phi,j)/j\\
    &\leq \sqrt{X}\left(206.16+26.72\log(M)+53.36\log(k+7.5)+(25.77+3.34\log(M))(\log(T)+1)\right.\\
    &\quad\qquad\left.+6.67\left\{\log(k+5.5)+\log(T)\log(k+4.5+T)\right\}\right),
\end{align*} for $T\geq 1$, where we have used the upper bounds of Lemma \ref{lem:Nbounds}, together with the estimates \begin{gather*}
    \sum_{j=1}^{T}\frac{1}{j}\leq\log(T)+1,\quad
    \sum_{j=1}^{T}\frac{\log(j+\alpha)}{j}\leq \log(T+\alpha)\log(T)+\log(\alpha+1),
\end{gather*} for $\alpha \geq 0$.

Next, note that for any choice of $\phi$, the set of trivial zeros (excluding a possible zero at $s=-\ell$) is contained in the set $\frac{1}{2}\mathbb{Z}\cap (-\infty,0)$, and the order of each zero is at most two. Thus, \begin{equation*}
    \sum_{\rho \text{ trivial}}\left\lvert\frac{{X}^{\rho}}{\rho+\ell}\right\rvert\leq 2\sum_{m=1}^{\infty}\frac{X^{-m/2}}{m/2}\leq 5,
\end{equation*} where the final inequality is valid if $X\geq 2$.

It remains to estimate the residue. We first suppose that $\phi=\mathrm{Sym}^2f$, $f\otimes \chi$ or $\chi$ with $\delta_\chi=1$. Then $L(\phi,0)$ is nonzero since $\Gamma(\phi,s)$ is regular at $s=0$. In that case, substituting $s=0$ in (\ref{eq:logderiv}) reveals \begin{equation*}
    -d\log L(\phi,0)=q_{\phi}+d\log \Gamma(\phi,0)+\sum_{\rho=\frac{1}{2}+i\gamma}\frac{2}{\gamma^2+4}.
\end{equation*} We know that $q_{\phi}\leq 2\log(M)$ and $d\log\Gamma(\phi,0)\leq 1$, and with Lemma \ref{lem:Nbounds} and the estimate \begin{equation*}
    \sum_{k=0}^{\infty} \frac{\log(j+\alpha)}{j^2+4}\leq \frac{5}{2}\log(\alpha+1)
\end{equation*} for $\alpha\geq 1$, we can show that the right hand side is bounded by \begin{equation*}
    24.47+5.05\log(M)+16.68\log(k+5.5).
\end{equation*} On the other hand, if $\phi=f\otimes g$ or $\chi$ with $\delta_{\chi}=0$, then $L(\phi,0)$ has a simple zero, so \begin{equation*}
    \underset{s=-\ell}{\mathrm{Res}}\left(d\log L(\phi,s)\frac{X^s}{s+\ell}\right)=\log(X)+\lim_{s\rightarrow 0}\left(d \log L(\phi,s)-as^{-1}\right)
\end{equation*} for some nonzero real number $a$. Using (\ref{eq:logderiv})) again, we find that \begin{equation*}
    \lim_{s\rightarrow 0}\left(-d \log L(\phi,s)+s^{-1}\right)=q_{\phi}+ \lim_{s\rightarrow 0}\left(-d \log \Gamma(\phi,s)+as^{-1}\right)+\sum_{\rho=\frac{1}{2}+i\gamma}\frac{2}{\gamma^2+4},
\end{equation*} and we first check that $q_{\phi}\leq 2\log(M)$ and the limit is bounded above by $\gamma$, the Euler-Mascheroni constant. As before, the sum over the zeros is bounded by \begin{equation*}
    24.05+5.05\log(M)+16.68\log(k+5.5),
\end{equation*} so upon choosing $T=\sqrt{X}$, we have proved \begin{lem}\label{lem:partialperron}
    For all choices of $\phi$ and $X\geq 4$, \begin{align*}
      & \left\lvert \frac{1}{2\pi i}\int_{\sigma_0-iT}^{\sigma_0+iT}d\log (\phi,s)\frac{X^s}{s+\ell}ds\right\rvert \leq 33.39+16.68\log(k+5.5)+5.38\log(M)\\
      &+\sqrt{X}\left(1230.84+29.15k+157.84\log(M)+3.78k\log(M)+60.03\log(k+7.5)\right)\\
      &+0.66\log(\sqrt{X}+k+3.25)\\
      &+\sqrt{X}\log(\sqrt{X}+k+4.5)\left(291.82+7.55k+1.67\log(M)+3.34\log(X)\right).
    \end{align*}
\end{lem} 
\section{Upper bounds for partial sums of Fourier coefficients}
In this section we relate the truncated Perron integrals from the last section to the usual Perron integrals, then to the Chebyshev theta functions associated to $\phi$, then to the sums appearing in Theorem \ref{thm:effectiveST} by Abel summation.

We begin by bounding the non-truncated Perron integral. \begin{lem}\label{lem:nontruncatedperron}
    For every $\phi \in \mathscr{L}(M)$, and $X\geq 17.33$, \begin{align*}
        &\left\lvert \frac{1}{2\pi i}\int_{\ell+\sigma_0-i\infty}^{\ell+\sigma_0+i\infty}d\log L(\phi,s-\ell)\frac{X^s}{s}ds\right\rvert \\ 
        &\leq X^{\ell+1/2}\log(X)\left(340.78+7.55k+1.67\log(M)+28.72\log(X)\right)\\
        &+X^{\ell+1/2}\left(1230.84+29.15k+157.84\log(M)+3.78k\log(M)+60.03\log(k+7.5)\right)\\
        &+31.86X^{\ell}\log(X)+X^{\ell}\left(33.39+16.68\log(k+5.5)+5.38\log(M)+31.2\log(X)\right)\\
        &+4X^{\ell-1/2}\log(X).
    \end{align*}
\end{lem}
\begin{proof}
    In light of Lemma \ref{lem:partialperron}, it suffices to bound \begin{equation}\label{eq:perrondifference}
        \left\lvert\frac{1}{2\pi i}\int_{\ell+\sigma_0-iT}^{\ell+\sigma_0+iT}d\log L(\phi,s-\ell)\frac{X^s}{s}ds-\frac{1}{2\pi i}\int_{\ell+\sigma_0-i\infty}^{\ell+\sigma_0+i\infty}d\log L(\phi,s-\ell)\frac{X^s}{s}ds\right\rvert 
    \end{equation} But recall that $\sigma_0=1+\log(X)^{-1}$, so if $\mathrm{Re}(s)=\sigma_0$ then \begin{equation*}
        d\log L(\phi,s)=\sum_{j=1}^{\infty}A_j(\phi)j^{-s},
    \end{equation*} with $\lvert A_j(\phi)\rvert \leq 4\Lambda(j)$ by the proof of Lemma \ref{lem:vonMangoldt}, where $\Lambda(j)$ is the von Mangoldt function. Thus the quantity in \ref{eq:perrondifference} is bounded above by \begin{equation}\label{eq:vMangmineq}
        4\sum_{j=1}^{\infty}\Lambda(j)j^{\ell}\left\lvert\frac{1}{2\pi i}\int_{\ell+\sigma_0-i\infty}^{\ell+\sigma_0+i\infty}\left(\frac{X}{j}\right)^s\frac{ds}{s}-\frac{1}{2\pi i}\int_{\ell+\sigma_0-iT}^{\ell+\sigma_0+iT}\left(\frac{X}{j}\right)^s\frac{ds}{s}\right\rvert.
    \end{equation} But \cite[Chapter 17]{davenport} tells us that the quantity at (\ref{eq:vMangmineq}) is bounded by \begin{equation*}
        4X^{\ell}\left\{\sum_{j=1}^{\infty}\Lambda(j)\left(\frac{X}{j}\right)^{\sigma_0}\min\{1, T^{-1}\lvert\log(X/j)\rvert^{-1}\}+\sigma_0T^{-1}\Lambda(X)\right\}+4\ell T^{-1}\Lambda(X)X^{\ell},
    \end{equation*} where the $\Lambda(X)$s are zero unless $X$ is an integer. Now the quantity in the large braces is bounded by \begin{equation*}
        6.35\sqrt{X}\log^2(X)+12.24\sqrt{X}\log(X)+\log(X)/\sqrt{X}+7.8\log(X)
    \end{equation*} in \cite[Lemma 4.1]{rousethorner}, and the claim follows.
\end{proof}
The next step is to show that we now have an upper bound for the Chebyshev theta functions of the second kind \begin{equation*}
    \psi(\phi,X)=\sum_{j\leq X}A_{\phi}(j)j^{\ell}.
\end{equation*}
\begin{lem}\label{lem:upperboundCh2}
    For any choice of $\phi$, and all $X\geq 17.33$, \begin{align*}
        &\left\lvert\psi(\phi,X)\right\rvert\leq X^{\ell+1/2}\log(X)\left(340.78+7.55k+1.67\log(M)+28.72\log(X)\right)\\
        &+X^{\ell+1/2}\left(1230.84+29.15k+157.84\log(M)+3.78k\log(M)+60.03\log(k+7.5)\right)\\
        &+33.86X^{\ell}\log(X)+X^{\ell}\left(33.39+16.68\log(k+5.5)+5.38\log(M)+31.2\log(X)\right)\\
        &+4X^{\ell-1/2}\log(X).
    \end{align*}
    \begin{proof}
        It follows from Perron's formula that \begin{equation*}
           \psi(\phi,X)=-\frac{1}{2\pi i}\int_{\mathrm{Re}(s)>\ell+1}d\log L( \phi,s-\ell)\frac{X^s}{s}ds+\frac{1}{2}A_{\phi}(X)X^{\ell},
        \end{equation*} where the $A_{\phi}(X)$ is zero unless $X$ is an integer. The claim then follows immediately from shifting the line of integration and applying Lemma \ref{lem:nontruncatedperron}.
    \end{proof}
\end{lem}
\begin{lem}\label{lem:auxbounds}
    For $X\geq \max\{6(k+1)^2,M,17\}$ and $k\geq 0$, \begin{equation}\label{eq:pi1klower}
        \pi_{1,k}(X)\geq \frac{1}{2(k+1)}X^{k+1}/\log(X),
    \end{equation} where the $2$ can be omitted if $k=0$. For $X\geq 17$ and $k\geq 0$, \begin{equation}\label{eq:pi1kupper}
        \pi_{1,k}(X)\leq 1.26X^{k+1}/\log(X).
    \end{equation}
    For all $X$, \begin{equation}\label{eq:piMk}
        \left\lvert\pi_{M,k}(X)-\pi_{1,k}(X)\right\rvert \leq 2.46\log(M)M^k.
    \end{equation}
    For all $Y\geq 2$ and $m\geq 0$, \begin{equation}\label{eq:shiftedCh2}
        \sum_{p\leq Y}p^m\log(p)\leq 2.01Y^{m+1}
    \end{equation}
\end{lem}
\begin{proof}
    To prove the lower bound in (\ref{eq:pi1klower}), we begin with the bound $|\pi(X)-\mathrm{li}(X)|\leq \sqrt{X}\log(X)$ (which follows from \cite{schoenfeldlowelltheta}), valid for $X\geq 2$. Applying Abel summation produces \begin{equation*}
        \pi_{1,k}(X)=X^k\pi(X)-k\int_{x}^{X}\pi(t)t^{k-1}dt,
    \end{equation*} where $x$ has been chosen so that $\mathrm{li}(x)=0$, and \begin{equation*}
        k\int_{x}^{X}\mathrm{li}(t)t^{k-1}dt=\mathrm{li}(X^{k+1}).
    \end{equation*} Hence, \begin{equation}\label{eq:lipi}
       \left\lvert\pi_{1,k}(X)-\mathrm{li}(X^{k+1})\right\rvert \leq X^{k+1/2}\log(X)+k\int_{x}^{X}\log(t)t^{k-1/2}dt.
    \end{equation} Now since $k\geq 4$, the right hand side of (\ref{eq:lipi}) is bounded above by $1.23 X^{k+1/2}\log(X)$, and $\mathrm{li}(X)\geq X/\log(X)$ for $X\geq 4$. The lower bound follows.

    The upper bound (\ref{eq:pi1kupper}) is much simpler and follows from Abel summation together with the upper bound $\pi(X)\leq 1.26X/\log(X)$ for $X\geq 17$ from \cite{schoenfeldrosseromega}.

    The upper bound (\ref{eq:piMk}) is a consequence of the bound $\omega(M)\leq 2.4573\log(M)$ of \cite{robin}.

    The upper bound (\ref{eq:shiftedCh2}) is easily established using Abel summation together with the bound $\sum_{p\leq Y}\log(p)\leq 1.01 Y$ of Rosser--Schoenfeld \cite{schoenfeldrosseromega}.
\end{proof}
Consider the Chebyshev functions of the first kind \begin{equation*}
    \theta(\phi,X)=\sum_{p\leq X}\hat{a}_p(\phi) p^{\ell}\log(p),
\end{equation*} where $\hat{a}_p(\phi)$ is the $p$th Fourier coefficient of $L(\phi,s)$ (with the analytic normalisation).\begin{lem}\label{lem:upperboundtheta}
    For any $\phi$, and $X\geq 17.33$, \begin{align*}
        &\left\lvert\theta(\phi,X)\right\rvert\leq X^{\ell+1/2}\log(X)\left(367.5+7.55k+1.67\log(M)+28.72\log(X)\right)\\
        &+X^{\ell+1/2}\left(1230.84+29.15k+157.84\log(M)+3.78k\log(M)+60.03\log(k+7.5)\right)\\
        &+33.86X^{\ell}\log(X)+X^{\ell}\left(33.39+16.68\log(k+5.5)+5.38\log(M)+31.2\log(X)\right)\\
        &+4X^{\ell-1/2}\log(X).
    \end{align*}
\end{lem}
\begin{proof}
    It suffices to give an upper bound for \begin{equation}\label{eq:differenceofChs}
        \left\lvert\psi(\phi,X)-\theta(\phi,X)\right\rvert.
    \end{equation} But since $A_{\phi}(p)=\hat{a}_{p}(\phi)$ and $A_{\phi}(j) \leq 4 \Lambda(j)$, the quantity at (\ref{eq:differenceofChs}) is bounded by \begin{equation*}
       4 \sum_{m\geq 2}\sum_{p\leq X^{1/m}}p^{\ell m }\log(p)\leq 26.72X^{\ell+1/2}\log(X),
    \end{equation*} by Lemma \ref{lem:auxbounds} and the claim follows by combining this with Lemma \ref{lem:upperboundCh2}.
\end{proof}

Finally, we set \begin{equation*}
    S(\phi,X)=\sum_{p\leq X}\hat{a}_{p}(\phi)p^{\ell},
\end{equation*} and we obtain the bound \begin{equation*}
    \lvert S(\phi,X)\rvert \leq X^{\ell+1/2}\log(X)\left(2595+87k+(248+6k)\log(M)\right),
\end{equation*} valid for $X\geq 17.33$, after combining Lemma \ref{lem:upperboundtheta} with the following formula, obtained by Abel summation: \begin{equation*}
    S(\phi,X)=\frac{\theta(\phi,X)}{\log(X)}+\int_{2}^{X}\frac{\theta(\phi,t)}{t\log^2(t)}dt
\end{equation*} We are now able to give upper bounds for the quantities in Theorem \ref{thm:effectiveST}. \begin{proof}[Proof of Theorem \ref{thm:effectiveST}]
    In all cases, \begin{equation*}
        \left\lvert\sum_{p\leq X}^{\prime}\hat{a}_{p}(\phi)p^{\ell}\right\rvert\leq S(\phi,X)+4\omega(M)M^{\ell},
    \end{equation*} where $\omega(M)$ is the number of distinct primes dividing $M$, and we have used the upper bound from the proof of (\ref{eq:piMk}). This implies that for $X\geq 17.33$, \begin{gather*}
        F(X), G(X) \leq \alpha X^{k/2}\log(X)+10 M^{(k-1)/2}\log(M)\\
        H(X) \leq \alpha X^{1/2}\log(X)+10 M\log(M),
    \end{gather*} where \begin{equation*}
        \alpha = 2595+87k+(248+6k)\log(M).
    \end{equation*} But as long as $X\geq M$, we can combine the terms on the right hand side of each estimate, as required.
\end{proof}

\section*{Acknowledgements}
It is a great pleasure to express my gratitude to Minhyong Kim for his guidance and patience during my work on this article. I am indebted to David Loeffler for pointing out the hypothesis contained within the SageMath source code, and making plain to me the connection between the vanishing of Fourier coefficients of modular forms and generators of Hecke algebras. Finally, I wish to thank Mike Eastwood for thinking about the problem with me when we first came across it, many years ago.

The author would also to acknowledge the support of the University of Warwick through its Chancellor's International Scholarship. In addition, the author is grateful for the visiting studentships provided by the University of Edinburgh and the University of Heriot--Watt.

\printbibliography

\end{document}